\documentclass{svmult}

\usepackage{url}
\usepackage{amssymb,latexsym,amsmath}
\usepackage{makeidx}         % allows index generation
\usepackage{graphicx}        % standard LaTeX graphics tool
\usepackage{multicol}        % used for the two-column index
\usepackage[bottom]{footmisc}% places footnotes at page bottom
\makeindex             % used for the subject index
\begin{document}

\title*{Sharp Global Bounds for the Hessian on Pseudo-Hermitian Manifolds}
\titlerunning{Hessian bounds}  % your contribution title if the original one is too long
\author{Sagun Chanillo \inst{1}\and
Juan J. Manfredi\inst{2}}
 \authorrunning{Chanillo-Manfredi}
\institute{Department of Mathematics, Rutgers University, 110 Frelinghuysen Rd., Piscataway, NJ 08854,
\texttt{chanillo@math.rutgers.edu}
\and Department of Mathematics,  University of Pittsburgh, Pittsburgh PA 15260,
 \texttt{manfredi@pitt.edu}}
\maketitle
\bigskip
\emph{Dedicated to the memory of  our friend and colleague Carlos Segovia.}
\bigskip

\section{Introduction}
\label{intro}
 In PDE theory, Harmonic Analysis enters in a fundamental way through the basic estimate valid for
 $f\in C_0^{\infty}(\mathbb{R}^n)$, which states,
\begin{equation}\label{intro:1}
\sum_{i,j=1}^n\left\| \frac{\partial^2 f}{\partial x_i \partial x_j}\right\|_{L^p(\mathbb{R}^n)} \le
 c(n,p) \left\| \Delta f\right\|_{L^p(\mathbb{R}^n)},\text{  for  }1<p<\infty.
\end{equation}
This estimate is really a statement of the $L^p$ boundedness of the
Riesz transforms, and thus (\ref{intro:1}) is a consequence of the
multiplier theorems of Marcinkiewicz and H\"ormander-Mikhlin,
\cite{stein}. More sophisticated variants  of (\ref{intro:1}) can be
proved by relying on the square function \cite{stein} and
\cite{segovia}. In particular (\ref{intro:1}) leads to a-priori
$W^{2,p}$ estimates for solutions of
\begin{equation}\label{intro:2}
\Delta u =f, \text{   for    } f\in L^p.
\end{equation}
Knowledge of $c(p,n)$ allows one to perform a perturbation of (\ref{intro:2}) and study
\begin{equation}\label{intro:3}
\sum_{i,j=1}^n a^{ij}(x)\frac{\partial^2 u}{\partial x_i \partial x_j} =f
\end{equation}
as was done by Cordes \cite{cordes}, where $A=(a^{ij})$ is bounded,
measurable, elliptic and close to the identity in a sense made
precise by Cordes. The availability of the estimates of
Alexandrov-Bakelman-Pucci and the Krylov-Safonov theory
\cite{gilbarg-trudinger} allows one to obtain estimates for
(\ref{intro:3}) in full generality without relying on a perturbation
argument. See also \cite{fanghualin}.\par

Our focus here will be to study the CR analog of (\ref{intro:3}). Since at this moment in time there
is no suitable Alexandrov-Bakelman-Pucci estimate for the CR analog of (\ref{intro:3}) we will
be seeking a perturbation approach  based on an analog of (\ref{intro:1}) on a  CR manifold. Our
main interest is the case $p=2$ in (\ref{intro:1}). In this case a simple integration by parts suffices
to prove (\ref{intro:1}) in $\mathbb{R}^n$. We easily see that for  $f\in C_0^{\infty}(\mathbb{R}^n)$ we have
\begin{equation}\label{intro:4}
\sum_{i,j=1}^n\left\| \frac{\partial^2 f}{\partial x_i \partial x_j}\right\|^2_{L^2(\mathbb{R}^n)} =
 \left\| \Delta f\right\|^2_{L^2(\mathbb{R}^n)}.
\end{equation}
In the case of (\ref{intro:1}) on a CR manifold a result has been
recently obtained by Domokos-Manfredi \cite{domokos-manfredi} in the
Heisenberg group. The proof in \cite{domokos-manfredi} makes uses of
the harmonic analysis techniques in the Heisenberg group developed
by Strichartz \cite{strichartz} that
 will not apply to studying such inequalities for the Hessian on a general CR manifold, although
 other nilpotent groups of step 2 can be treated similarly \cite{domokos-fanciullo}.
\par
 Instead we shall proceed by integration by parts and use of the Bochner technique. A Bochner
 identity on a CR manifold was obtained by Greenleaf \cite{greenleaf} and will play an
 important role in our computations.\par
 We now turn to our setup. We consider a smooth orientable manifold $M^{2n+1}$.
 Let $\mathcal{V}$ be a vector sub-bundle of the complexified tangent bundle
 $\mathbb{C}TM$. We say that $\mathcal{V}$ is a CR bundle if
 \begin{equation}\label{intro:5}
 \mathcal{V}\cap \overline{\mathcal{V}}=\{0 \},\text{    } [\mathcal{V},\mathcal{V}]\subset \mathcal{V}, \text{  and
 }\textrm{dim}_{\mathbb{C}}\mathcal{V}=n.
 \end{equation}
 A manifold equipped with a sub-bundle satisfying (\ref{intro:5}) will be called a CR manifold.
 See the book by Tr\`eves  \cite{treves}.
Consider the sub-bundle
 \begin{equation}\label{intro:6}
 H=\textrm{Re}\left(\mathcal{V} \oplus
 \overline{\mathcal{V}}\right).
 \end{equation}
 $H$ is a real $2n$-dimensional vector sub-bundle of the tangent bundle $TM$.  We assume that
the real line bundle $H^{\perp}\subset T^*M$, where $T^*M$ is the
cotangent bundle, has a smooth non-vanishing global section. This is
a choice of a non-vanishing  $1$-form $\theta$ on $M$ and
$(M,\theta)$ is said to define  a pseudo-hermitian structure.  $M$
is then called a pseudo-hermitian manifold. Associated to $\theta$
we have the Levi form $L_{\theta}$ given by
\begin{equation}\label{intro:7}
L_{\theta}(V, \overline{W})=-i\, d\theta(V\wedge\overline{W}), \text{   for   } V,W\in\mathcal{V}.
\end{equation}
We shall assume that $L_{\theta}$ is definite and orient $\theta$ by
requiring that $L_{\theta}$ is positive definite. In this case, we
say that $M$ is strongly pseudo-convex.  We shall always assume that
$M$ is strongly pseudo-convex.\par On a manifold $M$ that carries a
pseudo-hermitian structure, or a pseudo-hermitian manifold, there is
a unique vector field $T$, transverse to $H$ defined in
(\ref{intro:6}) with the properties
\begin{equation}\label{intro:8}
\theta(T)=1\text{
\ \ \    and\ \ \    } d\theta(T,\cdot)=0.
\end{equation}
$T$ is also called the Reeb vector field. The volume element
on $M$ is given by
\begin{equation}\label{intro:9}
dV=\theta\wedge(d\theta)^n.
\end{equation}
A complex valued $1$-form $\eta$ is said to be of type $(1,0)$ if
$\eta(\overline{W})=0$ for all  $W\in\mathcal{V}$, and of type
$(0,1)$ if $\eta(W)=0$ for all $W\in\mathcal{V}$.
\par
An admissible co-frame on an open subset of $M$ is a collection of
$(1,0)$ forms $\{\theta^1, \ldots,\theta^{\alpha},\ldots,\theta^n\}$
that locally form a basis for $\mathcal{V}^*$ and such that
$\theta^{\alpha}(T)=0$ for $1\le\alpha\le n$. We set
$\theta^{\overline{\alpha}}=\overline{\theta^{\alpha}}$. We then
have that $\{ \theta,\theta^{\alpha}, \theta^{\overline{\alpha}} \}$
locally form a basis of the complex co-vectors, and the dual basis
are the complex vector fields $\{T, Z_{\alpha},
\overline{Z_{\alpha}}\}$. For $f\in C^2(M)$ we set
\begin{equation}\label{intro:10}
Tf=f_0,\hspace{15pt} Z_\alpha f=f_\alpha,\hspace{15pt} \overline{Z_{\alpha}}f=f_{\overline{\alpha}}.
\end{equation}
We note that in the sequel all our functions  $f$ will be real valued.
\par
If follows from (\ref{intro:5}), (\ref{intro:7}), and  (\ref{intro:8}) that we can express
\begin{equation}\label{intro:11}
d\theta= i\,  h_{\alpha\overline{\beta}}\,\, \theta^{\alpha}\wedge
\theta^{\overline{\beta}}.
\end{equation}
The hermitian matrix $( h_{\alpha\overline{\beta}})$ is called the
Levi matrix. \par On pseudo-hermitian manifolds Webster
\cite{webster} has defined a connection, with connection forms
$\omega_{\alpha}^{\beta}$ and torsion forms $\tau_{\beta}=
A_{\beta\alpha}\theta^{\alpha}$, with structure relations
\begin{equation}\label{intro:12}
d\theta^{\beta}=\theta^{\alpha}\wedge \omega_{\alpha}^{\beta}+
\theta\wedge \tau_{\beta}, \hspace{20pt} \omega_{\alpha\overline{\beta}}+\omega_{\bar{\beta}\alpha}=dh_{\alpha\overline{\beta}}
\end{equation}
and
\begin{equation}\label{intro:13}
A_{\alpha\beta}=A_{\beta\alpha}.
\end{equation}
Webster defines a curvature form
$$\prod{}_{\alpha}^{\beta}=d\omega_{\alpha}^{\beta} -
\omega_{\alpha}^{\gamma}\wedge \omega_{\gamma}^{\beta},$$
where we have used the Einstein summation convention. Furthermore
in \cite{webster} it is shown that
$$\prod{}_{\alpha}^{\beta}= R_{\alpha\bar{\beta}\rho\bar{\sigma}}
\theta^\rho\wedge\theta^{\bar{\sigma}}+ \textrm{ other terms.}$$
Contracting two indices using the Levi matrix $(h_{\alpha\bar{\beta}})$
we get
\begin{equation}\label{intro:14}
R_{\alpha\bar{\beta}}= h^{\rho\bar{\sigma}}\,\,R_{\alpha\bar{\beta}\rho\bar{\sigma}}.
\end{equation}
The Webster-Ricci tensor $\textrm{Ric}(V,V)$ for $V\in\mathcal{V}$ is then
defined as
\begin{equation}\label{intro:15}
\textrm{Ric}(V,V)=R_{\alpha\bar{\beta}}x^{\alpha}\overline{x^\beta}, \text{  for  }
V=\sigma_{\alpha}x^{\alpha}Z_{\alpha}.\
\end{equation}
The torsion tensor is defined for $V\in \mathcal{V}$ as follows
\begin{equation}\label{intro:16}
\textrm{Tor}(V,V)=i  \left( A_{\bar{\alpha}\bar{\beta}}\overline{x^\alpha}
\bar{x}_\beta- A_{\alpha\beta}x^{\alpha}x^{\beta}\right).
\end{equation}
In \cite{webster}, Prop. (2.2), Webster proves that the torsion
vanishes if $\mathcal{L}_T$ preserves $H$, where $\mathcal{L}_T$ is
the Lie derivative. In particular if $M$ is a hypersurface in
$\mathbb{C}^{n+1}$ given by the defining function $\rho$
\begin{equation}\label{intro:17}
\textrm{Im} z_{n+1}=\rho(z, \overline{z}),
\hspace{20pt} z=(z_1,z_2,\ldots, z_n)
\end{equation}
then Webster's hypothesis is fulfilled and the torsion tensor
vanishes on $M$. Thus for the standard CR structure on the sphere
$S^{2n+1}$ and on the Heisenberg group the torsion vanishes.\par
Our main focus will be the sub-Laplacian $\Delta_b$. We define
the horizontal gradient $\nabla_b$ and $\Delta_b$ as follows:
\begin{equation}\label{intro:18}
\nabla_b f=\sum_{\alpha} f_{\overline{\alpha}} Z_{\alpha},
\end{equation}
\begin{equation}\label{intro:19}
\Delta_b f
=\sum_{\alpha}
f_{\alpha\bar{\alpha}}
+f_{\bar{\alpha}\alpha}.
\end{equation}
When $n=1$ we will need to frame our results in terms of the
CR Paneitz operator. Define the Kohn Laplacian $\square_b$ by
\begin{equation}\label{intro:20}
\square_b=\Delta_b+ i \, T.
\end{equation}
Then the CR Paneitz operator $P_0$ is defined by
\begin{equation}\label{intro:21}
P_0f= \left(\overline{\square}_b\square_b+\square_b\overline{\square}_b\right) f
-2 \left(Q+ \overline{Q}\right)f,
\end{equation}
where
$$Qf=2 i \, (A^{11}f_1)_1.$$
See \cite{lee} and \cite{hunglinchiu} for further  details.
\vfill\newpage
\section{The Main Theorem}\label{maintheorem}
\begin{theorem}\label{maintheorem}
Let $M^{2n+1}$ be a strictly pseudo-convex pseudo-hermitian manifold.
When $M$ is non compact assume that $f\in C_0^{\infty}(M)$. When  $M$ is compact
with $\partial M= \emptyset $ we may assume $f\in C^{\infty}(M)$.
When $f$ is real valued and $n\ge 2$ we have
\begin{equation}\label{main:a}
\sum_{\alpha,\beta} \int_M \!\! ||f_{\alpha\beta}||^2 +
||f_{\alpha\bar{\beta}}||^2 +\int_M\!\!  \left(\!
\textrm{Ric}+\frac{n}{2}\textrm{Tor}\right)\!\! (\nabla_b f,\nabla_b
f) \le\frac{(n+2)}{2n}\!\! \int_M \!\! |\Delta_b f|^2.
\end{equation}
When $n=1$ assume that the CR Paneitz operator $P_0\ge 0$. For
$f\in C_0^\infty(M)$ we then have
\begin{equation}\label{main:b}
\int_M \!\! ||f_{11}||^2 + ||f_{1\bar{1}}||^2 +\int_M\!\!
\left(\textrm{Ric}-\frac{3}{2}\textrm{Tor}\right)(\nabla_b
f,\nabla_b f) \le\frac{3}{2}\!\! \int_M \!\! |\Delta_b f|^2.
\end{equation}
Here by $\sum_{\alpha,\beta}||f_{\alpha\beta}||^2$ we mean the
Hilbert-Schmidt norm square of the tensor and similarly for
$\sum_{\alpha,\beta}||f_{\alpha\bar{\beta}}||^2$.
\end{theorem}
\begin{proof}
We begin by noting the Bochner identity established by Greenleaf,
Lemma 3 in \cite{greenleaf}:
\begin{eqnarray} \label{main:24}
\frac{1}{2}\Delta_b\left(|\nabla_b f|^2\right) & = & \sum_{\alpha,\beta}
 |f_{\alpha\beta}|^2 + |f_{\alpha\bar{\beta}}|^2+ \textrm{Re}\left(\nabla_bf,\nabla_b(\Delta_b f)\right)
 \\
  & + & \left(\textrm{Ric}+\frac{n-2}{2}\textrm{Tor}\right)(\nabla_b,\nabla_b)+ i
  \sum_{\alpha}\left(f_{\overline{\alpha}}f_{\alpha 0}- f_{\alpha}f_{\bar{\alpha} 0}\right). \notag
\end{eqnarray}
where for $V,W\in\mathcal{V}$ we use the notation
$(V,W)=L_{\theta}(V,\overline{W})$ and $|V|=(V,V)^{1/2}$. We have
also abused notation above and represented the Hilbert-Schmidt norm
of the tensor $f_{\alpha\beta}$ in terms of its expression in the
local frame which we will continue to do in the rest of the proof.
Using the fact that $f\in C_0^{\infty}(M)$ or if $\partial
M=\emptyset$, $M$ is compact, integrate (\ref{main:24}) over $M$
using the volume (\ref{intro:9}) to get
\begin{eqnarray}\label{main:25}
\int_M \sum_{\alpha,\beta}
 |f_{\alpha\beta}|^2   & + &|f_{\alpha\bar{\beta}}|^2 +
 \left(\textrm{Ric}+\frac{n-2}{2}\textrm{Tor}\right)(\nabla_b f,\nabla_b f)
 \\ & +&  i \int_M   \sum_{\alpha}\left(f_{\overline{\alpha}}f_{\alpha 0}- f_{\alpha}f_{\bar{\alpha} 0}\right)
 = -\int_M \textrm{Re}\left(\nabla_b f, \nabla_b(\Delta_b f)\right). \notag
\end{eqnarray}
Integration by parts in the term on the right yields (see (5.4) in \cite{greenleaf})
\begin{equation}\label{main:26}
-\int_M\textrm{Re}(\nabla_b f, \nabla_b(\Delta_b f))= \frac{1}{2}
\int_M |\Delta_b f|^2.
\end{equation}
Combining (\ref{main:25}) and (\ref{main:26}) we get
\begin{eqnarray}\label{main:27}
\int_M \sum_{\alpha,\beta}
 |f_{\alpha\beta}|^2   & + &|f_{\alpha\bar{\beta}}|^2+ \int_M
 \left(\textrm{Ric}+\frac{n-2}{2}\textrm{Tor}\right)(\nabla_b f,\nabla_b f)
 \\ & +&  i \int_M   \sum_{\alpha}\left(f_{\overline{\alpha}}f_{\alpha 0}- f_{\alpha}f_{\bar{\alpha} 0}\right)
 =  \frac{1}{2}
\int_M |\Delta_b f|^2.
 \notag
\end{eqnarray}
To handle the third integral in the left-hand side, we use Lemmas 4  and 5 of
\cite{greenleaf} (valid for real functions) according to which we have
\begin{equation}\label{main:28}
 i \int_M   \sum_{\alpha} (f_{\overline{\alpha}}f_{\alpha 0}- f_{\alpha}f_{\bar{\alpha} 0} )=\frac{2}{n} \int_M\!\! \bigg( \sum_{\alpha,\beta}\big( | f_{\alpha\bar{\beta}}|^2-|f_{\alpha\beta}|^2\big)-
 \textrm{Ric}(\nabla_b f, \nabla_b f)\!\!\bigg),
\end{equation}
and
\begin{eqnarray}\label{main:29}
 i \int_M   \sum_{\alpha} (f_{\overline{\alpha}}f_{\alpha 0}- f_{\alpha}f_{\bar{\alpha} 0} ) = & - &\frac{4}{n}
  \int_M \big| \sum_{\alpha}  f_{\alpha\bar{\alpha}}\big|^2\\
  & + &\frac{1}{n}\int_M |\Delta_b f|^2 \notag \\
 & + &\int_M
 \textrm{Tor}(\nabla_b f, \nabla_b f). \notag
\end{eqnarray}
Applying the Cauchy-Schwarz inequality to the the first term in the right-hand side of (\ref{main:29}) we get
\begin{eqnarray}\label{main:30}
 i \int_M   \sum_{\alpha} (f_{\overline{\alpha}}f_{\alpha 0}- f_{\alpha}f_{\bar{\alpha} 0} ) \ge & - &4
  \int_M  \sum_{\alpha, \beta}  |f_{\alpha\bar{\beta}}|^2\\
  & + &\frac{1}{n}\int_M |\Delta_b f|^2 \notag \\
 & + &\int_M
 \textrm{Tor}(\nabla_b f, \nabla_b f). \notag
\end{eqnarray}
Multiply (\ref{main:28}) by $1-c$ and (\ref{main:30}) by $c$  , $0<c<1$, and where $c$ will eventually
be chosen to be $1/(n+1)$, and add to get

\begin{eqnarray}\label{main:31}
 i \int_M   \sum_{\alpha} (f_{\overline{\alpha}}f_{\alpha 0}- f_{\alpha}f_{\bar{\alpha} 0} ) \ge & 2  &
 \frac{(1-c)}{n}
  \int_M  \sum_{\alpha, \beta} \left( |f_{\alpha\bar{\beta}}|^2- |f_{\alpha\beta}|^2\right)\\
  & - & 2\frac{(1-c)}{n} \int_M  \textrm{Ric}(\nabla_b f, \nabla_b f) \notag \\
  & - & 4c\int_M \sum_{\alpha,\beta} |f_{\alpha\overline{\beta}}|^2   \notag \\
  & + &\frac{c}{n}\int_M |\Delta_b f|^2
+c \int_M
 \textrm{Tor}(\nabla_b f, \nabla_b f). \notag
\end{eqnarray}

We now insert (\ref{main:31}) into (\ref{main:27})  and simplify. We  have
\begin{eqnarray}\label{main:32}
\left(1-\frac{2(1-c)}{n}\right) \int_M  \textrm{Ric}(\nabla_b f, \nabla_b f) &+&\notag\\
 \left(\frac{(n-2)}{2}+c\right) \int_M \textrm{Tor}(\nabla_b f, \nabla_b f)&+& \notag \\
 \left(1+\frac{2(1-c)}{n} - 4c\right) \int_M \sum_{\alpha,\beta} |f_{\alpha\overline{\beta}}|^2 &+&  \\
 \left( 1-\frac{2(1-c)}{n}\right) \int_M \sum_{\alpha,\beta} |f_{\alpha \beta}|^2 & \le &
 \left(\frac{1}{2}-\frac{c}{n}\right) \int_M |\Delta_b f|^2.    \notag
\end{eqnarray}
Let $c=1/(n+1)$. Then (\ref{main:32}) becomes
\begin{eqnarray}\label{main:33}
\left(\frac{n-1}{n+1}\right)\bigg[ \int_M  \sum_{\alpha,\beta} \big(|f_{\alpha\beta}|^2 +\!\!
|f_{\alpha\bar{\beta}}|^2 \big) & + &
\int_M \!\!\left(\textrm{Ric}+\frac{n}{2}\textrm{Tor}\right) (\nabla_b f,\nabla_b f)
\bigg]\\
& \le & \left(\frac{n-1}{n+1}\right)\left(\frac{n+2}{2n}\right) \int_M |\Delta_b f|^2. \notag
\end{eqnarray}
Since $n\ge 2$, $n-1>0$ and we can cancel the factor $\frac{n-1}{n+1}$ from both sides
to get (\ref{main:a}).\par
We now establish (\ref{main:b}) using some results by Li-Luk \cite{liluk} and \cite{hunglinchiu}. When $n=1$, identity (\ref{main:27}) becomes
\begin{eqnarray}\label{main:33}
\int_M
 |f_{1\bar{1}}|^2   & + &|f_{11}|^2+ \int_M
 \left(\textrm{Ric}-\frac{1}{2}\textrm{Tor}\right)(\nabla_b f,\nabla_b f)
 \\ & +&  i \int_M  \left(f_{10}f_{\bar{1}}-f_{\bar{1} 0}f_1\right)
 =  \frac{1}{2}
\int_M |\Delta_b f|^2.
 \notag
\end{eqnarray}
By (3.8) in \cite{liluk} we have
$$ i \int_M
\left(f_{01}f_{\bar{1}}-f_{0\bar{1} }f_1\right)= -\int_M f_0^2.$$
\end{proof}
Moreover, by (3.6) in \cite{liluk} we also have
$$i\left(f_{10}f_{\bar{1}}-f_{\bar{1} 0}f_1\right)=i \left(f_{01}f_{\bar{1}}-f_{0 \bar{1}}f_1\right)+
\textrm{Tor}(\nabla_b f,\nabla_b f)$$ and combining the last two identities we get
\begin{equation}\label{main:34}
i \int_M  \left(f_{10}f_{\bar{1}}-f_{\bar{1} 0}f_1\right)=  -\int_M f_0^2+\int_M \textrm{Tor}(\nabla_b f,\nabla_b f).
\end{equation}
Substituting (\ref{main:34}) into (\ref{main:33}) we obtain
\begin{eqnarray}\label{main:35}
\int_M
 |f_{1\bar{1}}|^2   +|f_{11}|^2  +  \int_M
 \left(\textrm{Ric}+\frac{1}{2}\textrm{Tor}\right) (\nabla_b f,\nabla_b f)
&-& \int_M  f_0^2 \\
&=&  \frac{1}{2}   \int_M |\Delta_b f|^2.\notag
\end{eqnarray}
Next, we use (3.4) in \cite{hunglinchiu},
\begin{equation}\label{main:36}
\int_M f_0^2=\int_M |\Delta_b f|^2+ 2\int_M\textrm{Tor}(\nabla_b f, \nabla_b f)-\frac{1}{2}\int_M
P_0f\cdot f.
\end{equation}
Finally,  substitute (\ref{main:36}) into (\ref{main:35}) and  simplify to get
\begin{eqnarray}
\int_M
 |f_{1\bar{1}}|^2   +|f_{11}|^2  +  \int_M
 \left(\textrm{Ric}-\frac{3}{2}\textrm{Tor}\right) (\nabla_b f,\nabla_b f)
  & + &   \frac{1}{2} \int_M
P_0f\cdot f  \notag \\
& = &\frac{3}{2}  \int_M |\Delta_b f|^2.\notag
\end{eqnarray}
Assuming $P_0\ge0$ we obtain (\ref{main:b}).
\qed \par\bigskip
We now wish to make some remarks about our theorem:\par\medskip
{\textbf{(a)}} It is shown in  \cite{domokos-manfredi} that on the Heisenberg group the constant $(n+2)/2n$ is sharp.
Since the Heisenberg group is a pseudo-hermitian manifold with $\textrm{Ric}\equiv0$ and
$\textrm{Tor}\equiv0$, we easily conclude our theorem is sharp and contains the result proved in \cite{domokos-manfredi}.\par\medskip
{\textbf{(b)}} We notice that when we consider manifolds such that
$\textrm{Ric}+(n/2)\textrm{Tor} >0$, then for $n\ge 2$, in general we have the strict inequality
$$\sum_{\alpha,\beta}\int_M |f_{\alpha\beta}|^2+|f_{\alpha\bar{\beta}}|^2 <
\frac{n+2}{2n}
\int_M |\Delta_bf|^2.$$ On the Heisenberg group  $\textrm{Ric}\equiv0$,
$\textrm{Tor}\equiv0$ and the constant $(n+2)/2n$ is achieved by a function with fast decay
\cite{domokos-manfredi}. Thus, the Heisenberg group is, in a sense, extremal for inequality
(\ref{main:a}) in Theorem \ref{maintheorem}. A similar remark holds for inequality ({\ref{main:b}).\par\medskip
  {\textbf{(c)}} The hypothesis on the Paneitz operator in the case $n=1$ in our theorem is satisfied
  on manifolds with zero torsion. A result from \cite{changchengchiu} shows that if the torsion vanishes
  the Paneitz operator is non-negative.\par\medskip
 {\textbf{(d)}} We note that Chiu \cite{hunglinchiu} shows how to perturb the standard pseudo-hermitian
 structure in $\mathbb{S}^3$ to get a structure with non-zero torsion, for which $P_0>0$ and
 $\textrm{Ric}-(3/2)\textrm{Tor}>1$. To get such a structure, let $\theta$ be the contact form
 associated to the standard structure on $\mathbb{S}^3$. Fix $g$ a smooth function on $\mathbb{S}^3$.
 For $\epsilon>0$ consider
 \begin{equation}\label{d:1}
 \tilde\theta= e^{2f} \theta, \text{   where   } f=\epsilon^3 \sin(\frac{g}{\epsilon}).
 \end{equation}
 Since the sign of the Paneitz operator  is a CR invariant and $\theta$ has zero torsion
we conclude by \cite{changchengchiu} that the CR Paneitz operator
$\tilde{P}_0$ associated to $\tilde\theta$ satisfies
$\tilde{P}_0>0$. Furthermore following the computation in Lemma
(4.7) of \cite{hunglinchiu}, we easily have for small $\epsilon$
that
$$\textrm{Ric}-\frac{3}{2}\textrm{Tor}\ge \left( 2+ O(\epsilon)\right) e^{-2f}\ge 1\ge 0.$$
Thus, the hypothesis of the case $n=1$ in our theorem are met, and for such $(M, \tilde\theta)$ we have,
for $f\in C^{\infty}(M)$ the estimate
$$\int_M |f_{11}|^2 + |f_{1\bar{1}}|^2 \, dV\le \frac{3}{2}\int_M |\Delta_b f|^2 \, dV.$$\par\medskip
{\textbf{(e)}} Compact pseudo-hermitian 3-manifolds with negative Webster curvature may be
constructed by considering the co-sphere bundle of a compact Riemann surface of genus $g$,
$g\ge 2$. Such a construction is given in \cite{chern-hamilton}.

\section{Applications to PDE}
For applications to subelliptic PDE it is helpful to re-state our main result Theorem \ref{maintheorem} in its
real version. We set
$$X_i=\textrm{Re}(Z_i)\text{   and    } X_{i+n}=\textrm{Im}(Z_i)$$
for $i=1,2\ldots,n$. The horizontal gradient of a function is the vector field
$$\mathfrak{X}(f)=\sum_{i=1}^{2n} X_i(f) X_i.$$
Its sublaplacian is  given by
$$\Delta_{\mathfrak{X}} f=\sum_{i=1}^{2n} X_iX_i(f),$$
and the horizontal second derivatives are the $2n\times2n$ matrix
$$ \mathfrak{X}^2f=\left(X_i X_j (f)\right).$$
For $f$ real we  have the following relationships
$$\nabla_b f=\mathfrak{X}(f)+i\left( \sum_{i=1}^n X_i(f) X_{i+n}-X_{i+n}(f)X_i\right),$$
$$\Delta_b f=2 \,\Delta_{\mathfrak{X}} f, $$ and
$$\sum_{\alpha,\beta}     |f_{\alpha\beta}|^2 + |f_{\alpha\bar{\beta}}|^2=2  \sum_{i,j} |X_iX_j(f)|^2=2 |\mathfrak{X}^2 f|^2,$$
where the expression on the extreme right is the Hilbert-Schmidt
norm square of the tensor taken by viewing the Levi form as a metric
on $H$.

\begin{theorem}\label{theorem2}
Let $M^{2n+1}$ be a strictly pseudo-convex pseudo-hermitian manifold.
When $M$ is non compact assume that $f\in C_0^{\infty}(M)$. When  $M$ is compact
with $\partial M= \emptyset $ we may assume $f\in C^{\infty}(M)$.
When $f$ is real valued and $n\ge 2$ we have
\begin{equation}\label{main2:a}
 \int_M  |\mathfrak{X}^2 f|^2+\int_M  \frac{1}{2} \left(\! \textrm{Ric}+\frac{n}{2}\textrm{Tor}\right) (\nabla_b f,\nabla_b f)
\le\frac{(n+2)}{n}\!\! \int_M \!\! |\Delta_{\mathfrak{X}} f|^2.
\end{equation}
When $n=1$ assume that the CR Paneitz operator $P_0\ge 0$. For
$f\in C_0^\infty(M)$ we then have
\begin{equation}\label{main2:b}
 \int_M  |\mathfrak{X}^2 f|^2+\int_M  \frac{1}{2} \left(\! \textrm{Ric}-\frac{3}{2}\textrm{Tor}\right) (\nabla_b f,\nabla_b f)
\le 3 \int_M \!\! |\Delta_{\mathfrak{X}} f|^2.
\end{equation}
\end{theorem}
Let $A(x)= (a_{ij}(x))$ a $2n\times 2n$   matrix. Consider the second order linear operator
in non-divergence form
\begin{equation}\label{pde:41}
\mathcal{A}u(x)=\sum_{i,j=1}^{2n} a_{ij}(x)X_i X_j u(x),
\end{equation}
where coefficients $a_{ij}(x)$ are bounded measurable functions in a domain $\Omega\subset M^{2n+1}$.
Cordes \cite{cordes} and Talenti \cite{talenti} identified the optimal condition expressing how
far $\mathcal{A}$ can be from the identity and still be able to understand (\ref{pde:41}) as
a perturbation of the case $A(x)=I_{2n}$, when the operator is just the
sublaplacian. This is the so called Cordes condition that roughly says that all eigenvalues
of $A$ must cluster around a single value.
\begin{definition} (\cite{cordes},\cite{talenti}, \cite{domokos-manfredi})
We say that $A$ satisfies the Cordes condition $K_{\varepsilon ,
\sigma}$ if there exists $\varepsilon \in (0,1]$ and $\sigma >0
$ such that
\begin{equation}\label{cordescon}
    0< \frac{1}{\sigma} \leq \sum_{i,j=1}^{2n} \, a_{ij}^{2} (x)
    \leq \frac{1}{2n-1+\varepsilon}
    \left( \sum_{i=1}^{2n} a_{ii} (x) \right)^{2}
\end{equation}
for a.~e. $x\in\Omega$.
\end{definition}
Let $c_n=\frac{(n+2)}{n}$ for $n\ge 2$ and $c_1=3$ the constants in
the right-hand sides of Theorem \ref{theorem2}. We can now adapt the
proof of Theorem 2.1 in \cite{domokos-manfredi} to get
\begin{theorem} \label{theorem3} Let $M^{2n+1}$ be a strictly pseudo-convex pseudo-hermitian manifold
such that $\textrm{Ric}+\frac{n}{2}\textrm{Tor} \ge 0$ if $n\ge 2$
and $\textrm{Ric}-\frac{3}{2}\textrm{Tor} \ge 0$, $P_0\ge 0$ if
$n=1$. Let $0 < \varepsilon \leq 1$, $\sigma >0$ such that $\gamma =
\sqrt{(1 - \varepsilon) c_{n}} < 1$ and  $A$ satisfies the Cordes
condition $K_{\varepsilon , \sigma}$. Then for all $u \in
C_0^{\infty} (\Omega )$ we have the \textit{a-priori} estimate
\begin{equation} \|\mathfrak{X}^2 u \|_{L^2} \leq
\sqrt{1+\frac{2}{n}} \; \frac{1}{1 - \gamma} \; \|\alpha
\|_{L^{\infty}} \| \mathcal{A} u\|_{L^2} \, ,
\end{equation}
 where
$$\alpha (x) = \frac{\langle A(x) , I \rangle}{||A(x)||^2}=\frac{\sum_{i=1}^{2n}a_{ii}(x)}{\sum_{i,j=1}^{2n}a_{ij}^2(x)} \, .$$
\end{theorem}

\begin{proof}
 We start from formula (2.7) in \cite{domokos-manfredi} which gives
$$\int_{\Omega} \left\vert \Delta_{\mathfrak{X}}u(x) - \alpha (x)
{\mathcal{A}} u(x) \right\vert^2 dx \leq (1-\varepsilon )
\int_{\Omega} |\mathfrak{X}u|^2 dx.
$$
We now apply Theorem \ref{theorem2} to get
$$\int_{\Omega} \left\vert \Delta_{\mathfrak{X}}u(x) - \alpha (x)
{\mathcal{A}} u(x) \right\vert^2 dx \leq (1-\varepsilon )
c_n\!\! \int_{\Omega}\!\! |\Delta_{\mathfrak{X}} f|^2.
$$
The theorem then follows as in \cite{domokos-manfredi}.
\qed\end{proof}
\subparagraph{Remark:} The hypothesis of Theorem \ref{theorem2}, $n\ge 2$,   can be weakened to assume
only a bound from below
$$ \textrm{Ric}+\frac{n}{2}\textrm{Tor} \ge -K, \text{   with }   K >0$$  to
obtain estimates of the type
\begin{equation}\label{remark}
 \int_M  |\mathfrak{X}^2 f|^2
\le\frac{(n+2)}{n}\!\! \int_M \!\! |\Delta_{\mathfrak{X}} f|^2 + 2 K\int_M |\mathfrak{X}f|^2.
\end{equation}
 A similar remark applies to the case $n=1$.\par

 We finish this paper by indicating how the \textit{a priori} estimate
 of Theorem \ref{theorem3} can be used to prove regularity for
 $p$-harmonic functions in the Heisenberg group $\mathcal{H}^n$when $p$ is close to $2$. We follow \cite{domokos-manfredi}, where
 full details can be found. Recall that,  for
 $1<p<\infty$,  a $p$-harmonic function $u$ in a domain $\Omega\subset\mathcal{H}^n$ is a function in
the horizontal Sobolev space
$$W_{\mathfrak{X}, \textrm{loc}}^{1,p} (\Omega)=\left\{
u\colon\Omega\mapsto\mathbb{R}\text{ such that }  u,\mathfrak{X}u\in L^p_{\textrm{loc}}(\Omega)\right\}$$ such that

 \begin{equation}\label{pde:45}
\sum_{i=1}^{2n} X_i \left( |\mathfrak{X} u |^{p-2} \, X_i u \right) = 0 \, ,
\; \mbox{in} \; \Omega \,
\end{equation}
in the weak sense. That is, for all $\phi\in C_0^{\infty}(\Omega)$ we have
\begin{equation}\label{pde:46}
\int_{\Omega } |\mathfrak{X}u (x)|^{p-2}(\mathfrak{X}u(x),
\mathfrak{X}\phi (x)\, dx =0.
\end{equation}
Assume for the moment that $u$ is a smooth solution of
(\ref{pde:45}). We can then differentiate to obtain
\begin{equation}\label{pde:47}
 \sum_{i,j=1}^{2n} \, a_{ij} \, X_i X_j u= 0 \, , \;
 \mbox{in} \; \Omega
 \end{equation}
where
$$a_{ij} (x) =   \delta_{ij} + (p-2)
\frac{X_i u (x) \, X_j u (x)}{ | \mathfrak{X}u (x) |^{2}} \,
.$$
A calculation shows that this matrix satisfies the Cordes
condition (\ref{cordescon}) precisely when
\begin{equation}\label{pde:49}
 p-2 \in \left( \frac{n - n \sqrt{4 n^2 + 4n -3}}{2n^2 + 2n -2}
 \, , \, \frac{n + n \sqrt{ 4 n^2 + 4n -3}}{2n^2 + 2n -2}\right).
 \end{equation}
In the case $n=1$ this simplifies to
$$p-2 \in \left( \frac{1 - \sqrt{5}}{2} \, , \,
\frac{1 + \sqrt{5}}{2} \right) \, .$$
We then deduce \textit{a priori} estimates for $\mathfrak{X}^2u$  from Theorem \ref{theorem3}. To apply the Cordes machinery to functions
that are only in $W_{\mathfrak{X}}^{1,p}$  we need to know that the
second derivatives $\mathfrak{X}^2u$ exist.  This is done in the Euclidean
case by a standard difference quotient argument applied to
a regularized $p$-Laplacian. In the Heisenberg case this would
correspond to proving that solutions to
\begin{equation}\label{pde:50}
\sum_{i=1}^{2n} X_i \left( \left( \frac{1}{m} + |\mathfrak{X} u |^2
\right)^{\frac{p-2}{2}} \, X_i u \right) = 0
\end{equation}
are smooth. Contrary to the Euclidean case (where solutions to
the regularized $p$-Laplacian are $C^{\infty}$-smooth) in the
subelliptic case this is  known only for $p\in [2, c(n))$ where
$c(n)=4$ for $n=1,2$, and $\lim_{n\to\infty}c(n)=2$
(see \cite{manfredimingione}.) The final result will combine the
limitations given  by (\ref{pde:49}) and $c(n)$.
\begin{theorem}\label{theorem4}
(Theorem 3.1 in \cite{domokos-manfredi})
For
$$2 \leq p < 2+ \frac{n + n \sqrt{
4n^2 + 4n -3}}{2n^2 + 2n -2} $$ we have that $p$-harmonic
functions in the Heisenberg group $\mathcal{H}^n$  are in
$W^{2,2}_{\mathfrak{X}, \textrm{loc}} (\Omega )$.
\end{theorem}
At least in the one-dimensional case $\mathcal{H}^1$ one can
also go below $p=2$. See Theorem 3.2 in \cite{domokos-manfredi}.
We also note that when $p$ is away from $2$, for example $p>4$ nothing
is known regarding the regularity of solutions to (\ref{pde:45}) or
its regularized version (\ref{pde:50}) unless we assume
\text{a  priori} that  the length of the gradient is bounded
below and above
$$0< \frac{1}{M} \le |\mathfrak{X}u| \le M < \infty.$$
See \cite{capogna} and \cite{manfredimingione}.

\begin{acknowledgement} S.C. supported in part by NSF Award DMS-0600971.
J.J.M. supported  in part  by NSF award
 DMS-0500983. S.C. wishes to thank Shri S. Devananda and Shri
 Raghavendra for encouragement in a difficult moment.
\end{acknowledgement}

\printindex
\end{document}